\documentclass[12pt]{amsart}

\usepackage{amsmath} 
\usepackage{amsfonts}
\usepackage{amssymb}
\usepackage{amsthm}
\usepackage{latexsym}
\usepackage{color}
\usepackage{enumerate}
\usepackage{mathrsfs}

\usepackage{setspace}

\usepackage{tikz}

\newtheorem{definition}{Definition}[section]
\newtheorem{lemma}{Lemma}[section]
\newtheorem{thm}{Theorem}[section]

\newtheorem{prop}{Proposition}[section]

\newtheorem{remark}{Remark}[section]
\numberwithin{equation}{section}

\usepackage{times}

\def\Ric{\textmd{Ric}}

\def\R{\mathbb{R}}

\def\R{\mathbb{R}}

\def\S{\Sigma}
\def\({\left(}
\def\){\right)}

\def\={\stackrel{(n=2)}{=}}
\def\p{\partial}

\newcommand{\be}{\begin{equation}}
\newcommand{\ee}{\end{equation}}
\newcommand{\bee}{\begin{equation*}}
\newcommand{\eee}{\end{equation*}}

\newcommand{\m}{\mathfrak{m}}

\allowdisplaybreaks

\begin{document}
	
	\title[]{Quasi-local Penrose inequalities with electric charge}
	
	\author[]{Po-Ning Chen}
	\address{Department of Mathematics, University of California, Riverside, USA}
	\email{poningc@ucr.edu}

	\author[]{Stephen McCormick}
	\address{Matematiska institutionen, Uppsala universitet, 751 06 Uppsala, Sweden}
	\email{stephen.mccormick@math.uu.se}

	\maketitle

	\begin{abstract}
The Riemannian Penrose inequality is a remarkable geometric inequality between the ADM mass of an asymptotically flat manifold with non-negative scalar curvature and the area of its outermost minimal surface. A version of the Riemannian Penrose inequality has also been established for the Einstein--Maxwell equations, where the lower bound on the mass also depends on the electric charge. In the context of quasi-local mass, one is interested in determining if, and for which quasi-local mass definitions, a quasi-local version of these inequalities also holds.
		
		It is known that the Brown--York quasi-local mass satisfies a quasi-local Riemannian Penrose inequality, however in the context of the Einstein--Maxwell equations, one expects that a quasi-local Riemannian Penrose inequality should also include a contribution from the electric charge. This article builds on ideas of Lu and Miao in \cite{LM} and of the first-named author in \cite{Chen18} to prove some charged quasi-local Penrose inequalities for a class of compact manifolds with boundary. In particular, we impose that the boundary is isometric to a closed surface in a suitable Reissner--Nordström manifold, which serves as a reference manifold for the quasi-local mass that we work with. In the case where the reference manifold has zero mass and non-zero electric charge, the lower bound on quasi-local mass is exactly the lower bound on the ADM mass given by the charged Riemannian Penrose inequality.
	\end{abstract}
	
	\section{Introduction}
	Motivated by the cosmic censorship conjecture, Penrose conjectured that an isolated system satisfying an appropriate energy condition should have mass bounded from below in terms of the horizon area of any black holes it contains \cite{Penrose}. While the full conjecture remains open, the Riemannian (or time-symmetric) case has been resolved in the affirmative and is a celebrated result in geometric analysis. It was proven by Huisken and Ilmanen \cite{H-I01} in the case of a connected horizon, and independently by Bray \cite{Bray01} allowing for the horizon to be disconnected. When gravity is coupled to electromagnetic fields -- the Einstein--Maxwell equations -- the lower bound on total mass should also contain a contribution from the electric charge; the charged Penrose inequality. The Riemannian case of this has also been settled. Under the assumption that the horizon is connected, this result follows from an old argument of Jang \cite{Jang79} when combined with the more recent development of weak inverse mean curvature flow by Huisken and Ilmanen \cite{H-I01} (see also \cite{Mc19}). The case of a disconnected horizon is more subtle when one includes electric charge and some extra care must be taken. However, an appropriate version of Bray's approach has been developed to account for the electric charge by Khuri, Weinstein and Yamada \cite{KWY-2017}, resolving the inequality in the case of disconnected horizons.

In addition to the Riemannian Penrose inequality for asymptotically flat manifolds, there has been recent interest in quasi-local versions of the Penrose inequality. Namely, one would like to bound the quasi-local mass of a region from below in terms of an outermost horizon area. As there are many candidates for a quasi-local mass definition in the literature, we briefly digress to mention the quasi-local mass definition that we focus on in this article. The Brown--York mass is a definition that arises from the Hamiltonian formulation of general relativity and is defined in terms of an isometric embedding into a reference manifold, which is taken to be Euclidean space. Variations of this, as well as the more modern Wang--Yau mass, have been studied using different spaces as the reference manifold (see \cite{CWWY}). When we refer to quasi-local mass with respect to a given reference manifold, it is this definition that we are considering. Some more details are given after the statement of Theorem \ref{thm-QLPenrose-C}.

A quasi-local Penrose inequality for the usual Brown--York mass was established by Shi and Tam \cite{ShiTam07} and by Miao \cite{Miao09}. However, one observes that for a surface enclosing a horizon, the obtained Penrose inequality is a strict inequality. This contrasts the Riemannian Penrose inequality for asymptotically flat manifolds where equality holds precisely for the Schwarzschild manifold. To obtain a sharp inequality where the equality holds for surfaces in the Schwarzschild manifold, Lu and Miao proved a Penrose inequality for a quasi-local mass with a Schwarzschild manifold as the reference manifold \cite{LM}. In \cite{LM1}, Lu and Miao proved that the equality holds for their quasi-local Penrose inequality if and only if the surface is in the Schwarzschild manifold. See also \cite{CZ,SWY}.

To include the effect of matter, the first-named author proved a quasi-local Penrose inequality for the quasi-local mass with reference to a more general static manifold, including the Reissner--Nordström manifold \cite{Chen18}. Nevertheless, while the reference manifolds were extended, the matter fields do not contribute to the inequality obtained in \cite{Chen18}. In particular, for a surface in a spacetime enclosing a non-vanishing matter, the inequality is always strict. The main goal of this article is to strengthen the inequality obtained in \cite{Chen18} to include the contribution from the electric charge. We remark that Alaee, Khuri and Yau \cite{AKY} also give different versions of the quasi-local Penrose inequality, including contributions due to both angular momentum and the electric field. Our inequality captures the full contribution of the electric field to the Penrose inequality. In particular, the equality case holds for the inequality we obtain for surfaces in the Reissner--Nordström manifold.

	The precise details of the charged Riemannian Penrose inequality that we require will be given in Section \ref{S-Setup}. However, the statement is essentially as follows. Given asymptotically flat time-symmetric initial data for the Einstein-Maxwell equations satisfying the dominant energy condition, with outermost minimal surface $\S_H$ then
	\be 
		\m_{ADM}\geq \(\frac{|\S_H|}{16\pi}\)^{\frac12}\(1+\frac{4\pi Q^2}{ |\S_H|}\),
\ee 
where $\m_{ADM}$ is the ADM mass, $|\S_H|$ denotes the area of $\S_H$, and $Q$ is the total electric charge. Usually, this inequality is stated with an additional hypothesis that the electric field is divergence-free. However, to establish our main results, we will need to employ versions of this inequality that hold when the electric field is not divergence-free. For the sake of exposition, we reserve the statements of these versions of the inequality for Section \ref{S-Setup}.

Our main results rely on the corner-smoothing technique of Miao \cite{Miao02} and the smoothing of the electric field as in \cite{AKY}. We therefore recall the following definition of a manifold with metric admitting corners along $\S$. Consider a manifold $M$ with a closed hypersurface $\S$ that encloses a compact region $\Omega$ in $M$.
	\begin{definition}
		A metric admitting corners along $\S$ is defined to be a pair $(g_{-}, g_{+})$, where $g_{-}$ and
		$g_{+}$ are $C^{2,\alpha}_{loc}$ metrics on $\Omega$ and 
		$M \setminus \overline{\Omega}$ respectively,
		that are $C^2$ up to the boundary and induce the same metric on
		$\S$.
	\end{definition}
	We say such a metric admitting corners is asymptotically flat if $g_{+}$ is asymptotically flat on $M \setminus \overline{\Omega}$ in the usual sense (see Section \ref{S-Setup}).
	As we are interested in electrically charged initial data sets, we give the following analogous definition.
	\begin{definition}
		A charged manifold admitting corners along a hypersurface $\S$ is defined as a tuple $(M,g_-,g_+,E_-, E_+)$ such that $(g_-,g_+)$ is a metric admitting corners along $\S$ on a manifold $M$, and  $E_{-}$ and
		$E_{+}$ are $C^{1,\alpha}_{loc}$ vector fields on $\Omega$ and 
		$M \setminus \overline{\Omega}$ respectively, up to the boundary.
	\end{definition}
	
	Along the corner, we are particularly interested in the normal projection of $E$, the flux of the electric field. For this reason, we introduce the notation
	\be 
		\Phi:=E\cdot\nu,
	\ee 
	where $\nu$ is the unit normal vector to the surface in question, usually the corner. We then employ the notation
	\be
	\Phi_\pm:=E_\pm \cdot \nu.
	\ee

In this article, we first prove three closely related versions of the charged Riemannian Penrose inequality on a charged asymptotically flat manifold admitting corners, related to each of the three different versions of the charged Riemannian Penrose inequality given in Section \ref{S-Setup}. Specifically, we show the following.
	\begin{thm}\label{thm-penrose-corners-A}
	Let $(M,g_-,g_+,E_-,E_+)$ be a charged manifold with corners such that the outermost minimal surface $\S_H$ of $M$ is contained in $\Omega$. Suppose 
	\be \label{eq-cornerconds-A}
	H_-  -  H_+  \ge  2 \left  |  \Phi_+  -  \Phi_- \right | \qquad \text{ and }\qquad Q_\infty^2 \leq \frac{|\S_H|}{4\pi},
	\ee 
	and on $\Omega$ and $M \setminus \overline{\Omega}$, we have
	\be \label{eq-energy-condition-A}
	R(g)\geq 2|E|^2 + 4 |\nabla \cdot E|
	\ee
	and $ \nabla \cdot E_+ $ is compactly supported. Then
	\be 
	\m_{ADM}\geq\( \frac{|\S_H|}{16\pi}\)^{1/2}\( 1+\frac{4\pi Q_\infty^2}{|\S_H|} \)
	\ee 
	where $Q_\infty$ is the total electric charge on $M$, measured at infinity.
\end{thm}
	\begin{thm}\label{thm-penrose-corners-B}
	Let $(M,g_-,g_+,E_-,E_+)$ be a charged manifold with corners such that the outermost minimal surface $\S_H$ of $M$ is contained in $\Omega$, which we assume is topologically the product of $\S$ and an interval. Suppose further 
	\be \label{eq-cornerconds-B}
	H_-  \geq  H_+  \qquad \text{ and }\qquad   \Phi_+  \leq  \Phi_-,
	\ee 
	and on $\Omega$ and $M \setminus \overline{\Omega}$, we have
	\be  \label{eq-energy-condition-B}
	R(g)\geq 2|E|^2  \qquad \text{ and }\qquad  \nabla\cdot E\leq0.
	\ee

	\be 
	\m_{ADM}\geq\( \frac{|\S_H|}{16\pi}\)^{1/2}\( 1+\frac{4\pi Q_\infty^2}{|\S_H|} \)
	\ee 
	where $Q_\infty \ge 0$ is the total electric charge on $M$, measured at infinity.
\end{thm}
\begin{thm}\label{thm-penrose-corners-C}
	Let $(M,g_-,g_+,E_-,E_+)$ be a charged manifold with corners such that the outermost minimal surface $\S_H$ of $M$ is contained in $\Omega$, which we assume is topologically the product of $\S$ and an interval. Suppose further 
	\be \label{eq-cornerconds-C}
	H_-  \geq  H_+  \qquad \text{ and }\qquad   \Phi_+  \geq  \Phi_-,
	\ee 
	and on $\Omega$ and $M \setminus \overline{\Omega}$, we have
	\be  \label{eq-energy-condition-C}
	R(g)\geq 2|E|^2  \qquad \text{ and }\qquad  \nabla\cdot E\geq0.
	\ee
	
	\be 
	\m_{ADM}\geq\( \frac{|\S_H|}{16\pi}\)^{1/2}\( 1+\frac{4\pi Q_H^2}{|\S_H|} \)
	\ee 
	where $Q_H\ge 0$ is the electric charge of $\S_H$.
\end{thm}
We briefly remark on the condition that $\Omega$ be a cylinder, appearing in the above two theorems. The proofs rely on approximating the manifold with corner with a sequence of smooth manifolds, and while we can control the area of the outermost horizon in the approximating sequence we cannot rigorously rule out the possibility that the outermost horizon in the approximating sequence becomes disconnected. The proofs then apply a version of the charged Riemannian Penrose inequality that follows from the inverse mean curvature flow approach, which requires that the outermost horizon be connected.

	As a consequence of the above theorems, we prove the following quasi-local versions of the charged Penrose inequality. Below, we refer to two convexity conditions without giving explicit details for the sake of exposition. Their precise definitions are given in Section \ref{S-Extensions}.
	\begin{thm} \label{thm-QLPenrose-A}
		Let $(\Omega,g)$ be a compact manifold with two boundary components, $\S_H$ and $\S$, a horizon $(H=0)$ component and an outer component with $H>0$, respectively. Let $E$ be a vector field on $\Omega$ representing the electric field, satisfying $R(g)\geq2|E|^2+4|\nabla\cdot E|$. Assume that $\S$ isometrically embeds in some Reissner--Nordstr\"om manifold with total mass  $\overline m$ and total charge $\overline Q$ where $\overline Q^2 \leq \frac{|\S_H|}{4\pi}$. Suppose that one of the following holds
		\begin{enumerate}[A.]
		\item the convexity condition \emph{($\dagger$)} and $\overline m > 0$, or
		\item the convexity condition \emph{($\dagger\dagger$)} and $\overline m=0$,
		\end{enumerate}
		 and assume in addition that
		\[
		H > 2 | \overline \Phi -  \Phi | 
		\]
		where $\overline \Phi$ is the electric flux in the Reissner--Nordstr\"om manifold. Then
		\be 
		\frac{1}{8\pi}\int_{\S}V(H_o-H+2 | \overline \Phi -  \Phi|  )\,d\S\geq \( \frac{|\S_H|}{16\pi} \)^{1/2}\( 1+\frac{4\pi \overline Q^2}{|\S_H|}\)-\overline m
		\ee 
		where $V$ is the static potential for the Reissner--Nordstr\"om manifold and $H_o$ is the mean curvature of the isometric embedding.
	\end{thm} 	
\begin{remark}
The convexity condition \emph{($\dagger$)}, which was introduced in \cite{Chen18}, is given by Definition \ref{def-convex1}. The convexity condition \emph{($\dagger\dagger$)} is a new condition closely related to \emph{($\dagger$)}, and is given by Definition \ref{def-convex2}.
\end{remark}
	\begin{thm} \label{thm-QLPenrose-B}
	Let $(\Omega,g)$ be a compact manifold that is topologically a cylinder with boundary components, $\S_H$ and $\S$; a horizon $(H=0)$ component and an outer component with $H>0$, respectively. Let $E$ be a vector field on $\Omega$ representing the electric field, satisfying $R(g)\geq2|E|^2$ and $\nabla\cdot E\leq 0$. Assume that $\S$ isometrically embeds in some Reissner--Nordstr\"om manifold of charge $\overline Q\geq 0$ with mean curvature $H_o$ such that one of the following holds
		\begin{enumerate}[A.]
		\item the convexity condition \emph{($\dagger$)} and $\overline m > 0$, or
		\item the convexity condition \emph{($\dagger\dagger$)} and $\overline m=0$,
		\end{enumerate}
	and assume in addition that
	\[
		\overline\Phi \leq  \Phi
	\]
	where $\overline \Phi$ is the electric flux in the Reissner--Nordstr\"om manifold. Then
	\be 
	\frac{1}{8\pi}\int_{\S}V(H_o-H)\,d\S\geq \( \frac{|\S_H|}{16\pi} \)^{1/2}\( 1+\frac{4\pi \overline Q^2}{|\S_H|}\)-\overline m
	\ee 
	where $V$ is the static potential for the Reissner--Nordstr\"om manifold, and $\overline m$ is its mass.
\end{thm} 	

	\begin{thm} \label{thm-QLPenrose-C}
	Let $(\Omega,g)$ be a compact manifold that is topologically a cylinder with boundary components, $\S_H$ and $\S$; a horizon $(H=0)$ component and an outer component with $H>0$, respectively. Let $E$ be a vector field on $\Omega$ representing the electric field, satisfying $R(g)\geq2|E|^2$, $Q(\S_H)\geq0$, and $\nabla\cdot E\geq 0$. Assume that $\S$ isometrically embeds in some Reissner--Nordstr\"om manifold with mean curvature $H_o$ such that one of the following holds
		\begin{enumerate}[A.]
		\item the convexity condition \emph{($\dagger$)} and $\overline m > 0$, or
		\item the convexity condition \emph{($\dagger\dagger$)} and $\overline m=0$,
		\end{enumerate}
	and assume in addition that
	\[
	\overline\Phi \geq  \Phi
	\]
	where $\overline \Phi$ is the electric flux in the Reissner--Nordstr\"om manifold. Then
	\be \label{eq-penroseC}
	\frac{1}{8\pi}\int_{\S}V(H_o-H)\,d\S\geq \( \frac{|\S_H|}{16\pi} \)^{1/2}\( 1+\frac{4\pi Q(\S_H)^2}{|\S_H|}\)-\overline m
	\ee 
	where $V$ is the static potential for the Reissner--Nordstr\"om manifold and $\overline m$ is its mass.
\end{thm} 	

The quantity given on the right-hand side of \eqref{eq-penroseC} is the quasi-local mass with respect to a static reference manifold, discussed above. In this instance, we use a Reissner--Nordstr\"om manifold as the reference manifold. This quantity goes by several names in the literature, such as a weighted Brown--York mass integral or the Wang--Yau energy with respect to a static reference. One may view this quantity as somehow measuring how far a domain deviates from the given reference manifold, since this quasi-local mass trivially yields zero if the surface lies in the reference manifold and the isometric embedding is the identity map. Note that if the reference manifold has non-zero mass, then in general the quasi-local mass defined with respect to it will not be positive. However, we would like to remark that this is not the case when the reference manifold is taken to be a Reissner--Nordstr\"om manifold with zero mass. One can see from \eqref{eq-penroseC} that when $\overline m=0$, the quasi-local mass not only is positive, but the lower bound is exactly the lower bound for the charged Riemannian Penrose inequality. Furthermore, when $\overline m=0$, \eqref{eq-penroseC} is in fact a strict inequality, analogous to the quasi-local Penrose inequality for the Brown--York mass (see Remark \ref{rem-equality}). 

The outline of this article is as follows. Section \ref{S-Setup} briefly recalls some standard definitions and the specific versions of the charged Riemannian Penrose inequalities that we will make use of. In Section \ref{S-corners} we prove Theorems \ref{thm-penrose-corners-A}, \ref{thm-penrose-corners-B} and \ref{thm-penrose-corners-C}. Then Section \ref{S-Extensions} describes a charged Shi--Tam type of extension following \cite{Chen18}, which we use in Section \ref{S-proof} to finally prove Theorems \ref{thm-QLPenrose-A}, \ref{thm-QLPenrose-B} and \ref{thm-QLPenrose-C}.

\medskip

\noindent {\bf Acknowledgements.} This work was commenced while both of the authors were visiting Institut Mittag-Leffler in Djursholm, which is supported by the Swedish Research Council grant no. 2016-06596. The authors would like to thank Institut Mittag-Leffler for their hospitality during this time. P. Chen is supported by the Simons Foundation collaboration grant \#584785.

	\section{Setup and definitions}\label{S-Setup}
	In this section, we recall some standard definitions and explicitly give the versions of the charged Riemannian Penrose inequality with charged matter that we will need in what follows.
	
	A time-symmetric, charged initial data set consists of the triple $(M,g,E)$ where $(M,g)$ is assumed to be a Riemannian $3$-manifold and $E$ is a vector field representing the electric field. Generally one also imposes the dominant energy condition on such initial data, which in the time-symmetric case is simply the condition
	\be
	R(g) \ge 2 |E|^2
	\ee
	where $R(g)$ is the scalar curvature of $g$. Physically, this condition corresponds to the assumption that any physical matter has non-negative local energy density. Throughout, we are interested in asymptotically flat initial data sets, which physically represent isolated gravitating systems. We say $(M,g,E)$ is asymptotically flat (of order $\rho$, with one end) if:

	\begin{enumerate}
		\item $M$, after removing a compact set, is diffeomorphic to $\R^3$ minus the closed unit ball;
		\item $R(g),\nabla\cdot(E)\in L^1(M)$;
		\item in the Euclidean coordinates near infinity given by the diffeomorphism we have $|g-\delta|+|x|\,|\p g|+|x|^2\,|\p^2 g|=O(|x|^{\rho})$;
		\item and, $|E|+|x|\,|\p E|=O(|x|^{\rho-1})$
	\end{enumerate}
	where $\rho\in(-\frac12,-1]$. This is precisely what is required for the ADM mass and total electric charge to be well-defined. In the standard rectangular Cartesian coordinates near infinity, the ADM mass can be expressed as
	\be 
		\m_{ADM}=\lim\limits_{r\to\infty}\frac{1}{16\pi}\int_{\S_r}\(\p_i g_{ij}-\p_ig_{ij}\)\nu^j\,d\S,
	\ee 
	where the limit is taken over large coordinate spheres $\S_r$, $\nu^j$ is the unit normal, and repeated indices are summed over. Given any closed 2-surface $\Sigma$, the charge enclosed by $\Sigma$, denoted $Q(\Sigma)$, is defined by the flux integral
	\be
	Q(\Sigma) = \frac{1}{4 \pi} \int_{\Sigma} E_i \nu^i \, d\S = \frac{1}{4 \pi} \int_{\Sigma} \Phi \, d\S.
	\ee
	The total charge, $Q_\infty$ then can be expressed as
	\be 
	Q_\infty = \lim\limits_{r\to\infty}\frac{1}{4 \pi} \int_{\S_r} \Phi \, d\S_r.
	\ee 
	
We are now ready to state the three versions of the charged Riemannian Penrose inequality that we require.
	
	\begin{thm}[\cite{KWY-exts-2015,Mc19}]\label{thm-A}
		Let $(M,g,E)$ be a charged asymptotically flat 3-manifold satisfying $R(g)\geq 2|E|^2+4|\nabla\cdot E|$, containing an outermost minimal surface $\S$, and assume that $\nabla\cdot E$ is compactly supported. Then we have
			\be \label{eq-A}
			\m_{ADM}\geq \(\frac{|\S|}{16\pi}\)^{\frac12}\(1+\frac{4\pi Q_\infty^2}{ |\S|}\).
			\ee
		Furthermore, equality holds if and only if $(M,g,E)$ is a Reissner--Nordstr\"om manifold.
	\end{thm} 
\begin{remark}
	Theorem \ref{thm-A} is essentially Theorem 1.3 of \cite{KWY-exts-2015}, after applying a minor correction, as discussed in Section 3.3 of \cite{Mc19}. The correction relates to the hypothesis $R(g)\geq 2|E|^2+4|\nabla\cdot E|$, which is stronger than the usual dominant energy condition, sometimes called the enhanced dominant energy condition.
\end{remark}

		\begin{thm}[\cite{Jang79,Mc19}]\label{thm-B}
		Let $(M,g,E)$ be a charged asymptotically flat 3-manifold satisfying the dominant energy condition, containing a connected outermost minimal surface $\S$, and assume that exterior to $\S$ it holds that $Q_\infty\nabla\cdot E\leq0$.

		Then
		\be \label{eq-B}
		\m_{ADM}\geq \(\frac{|\S|}{16\pi}\)^{\frac12}\(1+\frac{4\pi Q_\infty^2}{ |\S|}\).
		\ee 
		Furthermore, equality holds if and only if $(M,g,E)$ is a Reissner--Nordstr\"om manifold.
	\end{thm}

\begin{thm}[\cite{Jang79,Mc19}]\label{thm-C}
	Let $(M,g,E)$ be a charged asymptotically flat 3-manifold satisfying the dominant energy condition, containing a connected outermost minimal surface $\S$, and assume that exterior to $\S$ it holds that $Q_\S\nabla\cdot E\geq0$.
	
	Then
		\be \label{eq-C}
		\m_{ADM}\geq \(\frac{|\S|}{16\pi}\)^{\frac12}\(1+\frac{4\pi Q_\S^2}{ |\S|}\),
		\ee 
		where $Q_\S=Q(\S)$ is the charge of $\S$. Furthermore, equality holds if and only if $(M,g,E)$ is a Reissner--Nordstr\"om manifold.
	\end{thm}

\section{The charged Riemannian Penrose inequality with corners}\label{S-corners}
	In this section, we recall Miao's corner-smoothing technique that allows one to smooth out Riemannian metrics while preserving non-negativity of scalar curvature. In addition to smoothing the metric, we simultaneously smooth the electric field following the method by Alaee, Khuri and Yau in \cite{AKY} to ensure that the hypotheses of the charged Riemannian Penrose inequality hold. The condition $H_-\geq H_+$ imposed on the corner amounts to insisting that the scalar curvature does not have a sharp drop in the distributional sense; that is, the charged dominant energy condition is preserved in a distributional sense along the corner. Similarly, the quantity $\Phi_+-\Phi_-$ having an appropriate sign corresponds to preserving the sign of $\nabla\cdot E$ distributionally across the corner. In order to preserve the energy condition required in Theorem \ref{thm-A}, one instead can ask that the quantities at the corner satisfy $H_-  -  H_+  \ge  2 \left | \Phi_+  -  \Phi_- \right |$.
	
	Denote by $(M, g_-,g_+,E_-,E_+)$ the charged manifold with corners that we wish to smooth out. To achieve this, we first smooth out the metric using a standard mollifier as in Proposition 3.1 of \cite{Miao02} to obtain a family of $C^2$ metrics $g_\delta$ on $M$ that is isometric to the metric admitting corners outside of a neighbourhood $\Sigma \times (-\delta , \delta)$ of $\S$. In the Gaussian neighbourhood $\Sigma \times (-\delta , \delta)$, the metric is of the form
	\be 
	g_\delta=\sigma_\delta(s)_{ab}dx^adx^b+ds^2
	\ee
	where $\sigma_\delta$ denotes the metrics on each leaf of constant $s$. We then smooth out the vector field $E$ as in Lemma 5.1 of \cite{AKY}. The proofs of our different versions of the charged Riemannian Penrose inequality with corners are almost identical, with only minor differences. We focus first on proving the version of the inequality based on Theorem \ref{thm-A}, then explain the differences that lead us to Theorems \ref{thm-B} and \ref {thm-C}.

	\begin{lemma} \label{lemma_energy_condition}
		Assuming  $H_-  -  H_+  \ge  2 \left | \Phi_+  -  \Phi_- \right |$, then there exists a $C^{1,\alpha}$ electric field vector $E_\delta$ on $ M$ that is uniformly bounded independent of $\delta$ such that 
		\be
		\begin{split}\label{eq-weakeDEC}
			R(g_\delta) - 2 |E_\delta|^2_{g_\delta}  - 4 |\nabla_{g_\delta} \cdot E_\delta  | \ge & h_\delta \\
		\end{split}
		\ee
		where $h_\delta$ is uniformly bounded on $ M$ and vanishes outside of $\Sigma \times (-\delta , \delta)$. Furthermore $E_\delta$ is exactly equal to $E$ outside of $\Sigma \times (-\delta , \delta)$.
	\end{lemma}
	\begin{proof}
		Let  $\phi\in C^{\infty}_{c}((-1,1))$ be the standard mollifier with $0\leq \phi(s)\leq 1$ and $\int_{-1}^1\phi ds=1$, used in the smoothing process defining $g_\delta$. It is shown in \cite{Miao02} that the scalar curvature of the smoothed metric satisfies
		\begin{equation}
		R(g_\delta)=O(1),\qquad (s,x)\in\left\{\frac{\delta^2}{100}<|s|\leq\frac{\delta}{2}\right\}\times \Sigma,
		\end{equation}
		\begin{equation}
		R(g_\delta)=O(1)+\left(H_-  -  H_+  \right)\frac{100}{\delta^2}
		\phi\left(\frac{100s}{\delta^2}\right),\qquad (s,x)\in\left[-\frac{\delta^2}{50},\frac{\delta^2}{50}\right]\times \Sigma,
		\end{equation}
		and by the construction in Lemma 5.1 of \cite{AKY}, the smoothed out vector field satisfies
		\begin{equation}
		\nabla_{g_\delta} \cdot E_\delta =O(1),\qquad (s,x)\in\left\{\frac{\delta^2}{50}<|s|\leq\frac{\delta}{2}\right\}\times \Sigma,
		\end{equation}
		and
		\begin{equation} \label{eq_div_estimate}
		\nabla_{g_\delta} \cdot E_\delta
		=\left(\Phi_+  -  \Phi_-\right)\frac{50}{\delta^2}
		\phi\left(\frac{100s}{\delta^2}\right)
		+O(1),\qquad (s,x)\in\left[-\frac{\delta^2}{50},\frac{\delta^2}{50}\right]\times \Sigma.
		\end{equation}
		One then sees that the condition $H_-  -  H_+  \ge  2 |  \Phi_+  -  \Phi_-  | \ge 0$ implies the conclusion, where $h_\delta$ arises from the $O(1)$ terms.
	\end{proof}

Note that in the above lemma, the smoothed initial data does not quite satisfy the energy condition that we wish to preserve. Nevertheless, we may perform a conformal change following \cite{Miao02,SY79} to obtain smoothed data that does indeed satisfy the energy condition. In what follows we will use the superscripts $+$ and $-$ to respectively denote the positive and negative part of a function. That is, for any function $f$ we write $f=f^+-f^-$ where $f^\pm\geq0$. The required conformal factor comes from solving an equation of the form
\be  \label{eq-SY-lemma}
	\left\{ \begin{array}{rrl}
	{\Delta}_{g} u -f^- u & = & 0 \\
	\lim_{x \rightarrow \infty} u & = & 1 \  ,
\end{array} 
\right.
\ee 
for some given $f$.

It follows from Lemma 3.2 of \cite{SY79} that \eqref{eq-SY-lemma} has a unique solution on an asymptotically flat manifold $M$ (without boundary) provided that $f^-$ is sufficiently small in $L^{3/2}(M)$. Furthermore, the solution has the expansion $f=1+\frac{A}{|x|}+O(|x|^{-2})$. For this reason, we reflect the manifold $ M$ across the minimal surface to obtain a manifold without boundary, $\widetilde M$.

In \cite{Miao02}, \eqref{eq-SY-lemma} was used to obtain an appropriate conformal factor, with $f^-$ taken to be proportional to the negative part of the scalar curvature $R(g_\delta)$. However, here we would like to make two different choices of $f^-$; one for each energy condition that we would like to preserve. Nevertheless, the proof of Proposition 4.1 of \cite{Miao02} shows that provided the $L^{6/5}$ and $L^3$ norms of $f^-$ are controlled by $\delta$ then the solution $u_\delta$ on $\widetilde{M}$ is uniformly close to $1$. This in turn implies that a metric of the form $\widetilde g_\delta:=u_\delta^4 g_\delta$ converges to $g$ in $C^0$ on $\widetilde{M}$ and in $C^2$ away from the corner. In addition, the proof of Lemma 4.2 of \cite{Miao02} then demonstrates that the ADM mass of $\widetilde g_\delta$ converges to the ADM mass of $g_+$.

From this, we obtain the following lemmas by making suitable choices of $f^-\in L^{6/5}(\widetilde M)\cap L^{3/2}(\widetilde M)\cap L^{3}(\widetilde M)$ to ensure that the conformally transformed metrics satisfy the energy conditions.
\begin{lemma}\label{lem-conf-changed-data}
	Assume $H_- - H_+\geq 2|\Phi_- - \Phi_+|$ and away from the corner \eqref{eq-energy-condition-A} holds. Then there exists a sequence of initial data $(\widetilde{M},\widetilde g_\delta, (1-\epsilon)^2\widetilde E_\delta)$ satisfying \eqref{eq-energy-condition-A}. Furthermore we have that $\widetilde g_\delta$ converges to $(g_-,g_+)$ in $C^0$ on $\widetilde{M}$ and $C^2$ away from the corner, and $\widetilde E_\delta$ converges to $(E_-,E_+)$ in $C^0$ on $ M$ and $C^1$ away from the corner.
\end{lemma}
\begin{proof}
	Consider the sequence of smoothed metrics $g_\delta$ above, and let $E_\delta$ be the vector field given by Lemma \ref{lemma_energy_condition}. Now set $f=\frac18 R(g_\delta)-\frac14 |E_\delta|^2_{g_\delta}-\frac12|\nabla\cdot E_\delta|$ in \eqref{eq-SY-lemma} and solve this on the doubled manifold, $\widetilde{M}$. Note that Lemma \ref{lemma_energy_condition} implies $f^-\in L^{6/5}(\widetilde{M})\cap L^{3/2}(\widetilde{M})\cap L^{3}(\widetilde{M})$, since $h_\delta$ therein is compactly supported.
	
	Let $u_\delta$ be the unique solution to \eqref{eq-SY-lemma} and consider the conformally transformed metric $\widetilde g_\delta=u^4g_\delta$ and vector field $\widetilde E_\delta=u^{-6} E_\delta$. Note that we have
	\be 
	\nabla_{\widetilde g_\delta}\cdot \widetilde E_\delta=u^{-6}\frac{1}{\sqrt g_\delta}\p_i(u^6\sqrt{g_\delta}u^{-6}E_\delta)=u^{-6}\nabla_{g_\delta}\cdot E_\delta,
	\ee 
	so the sign of the divergence of the electric field is preserved under the conformal transformation.
	
	By the usual transformation of scalar curvature formula, we have
	\be
	\begin{split}
	R(\widetilde g_\delta) - 2 &| (1-\epsilon)^2\widetilde E_{\delta}|^2_{\widetilde g_\delta} -4|(1-\epsilon)^2\nabla_{\widetilde g_\delta}\cdot \widetilde E_\delta|\\
	&=  u^{-5}_{\delta} \left ( R(g_\delta) u_{\delta} - 8{\Delta}_{g_\delta} u_\delta - 2(\frac{1-\epsilon}{ u_\delta})^4 |E_\delta|_{g_\delta}^2 u_\delta-4(\frac{1-\epsilon}{ u_\delta})^2|\nabla_{g_\delta}\cdot E_\delta|u_\delta \right),
	\end{split}
	\ee
	for some small $\epsilon>0$. We now choose $\delta$ sufficiently small so as to ensure $|u_\delta-1|<\epsilon$, so that we have 
	\be
	\begin{split}
	R(\widetilde g_\delta) - 2 &| (1-\epsilon)^2\widetilde E_{\delta}|^2_{\widetilde g_\delta} -4|(1-\epsilon)^2\nabla_{\widetilde g_\delta}\cdot \widetilde E_\delta|\\
	&=  u^{-5}_{\delta} \left ( R(g_\delta) u_{\delta} - 8{\Delta}_{g_\delta} u_\delta - 2 |E_\delta|_{g_\delta}^2 u_\delta-4|\nabla_{g_\delta}\cdot E_\delta|u_\delta \right),
\end{split}	
\ee
	which is non-negative since $u_\delta$ satisfies \eqref{eq-SY-lemma} with $f$ chosen as above.

\end{proof}

We would like to apply Theorem \ref{thm-A} to our smoothed initial data to conclude Theorems \ref{thm-penrose-corners-A}. However, it may be that after applying our conformal factor that $\S_H$ is no longer the outermost horizon. This can be circumvented by arguments of Miao (pages 279 -- 280 of \cite{Miao02}; see also Appendix A of \cite{MM19}) demonstrating that after such a conformal change the area of the outermost horizon can be made arbitrarily close to the area of $\S_H$ by choosing $\delta$ sufficiently small. In particular, if $\S_\delta$ is the outermost minimal surface in $(M,\widetilde g_\delta)$ containing $\S_H$ then, passing to a subsequence $\delta_i$ with $\delta_i\to 0$ as $i\to\infty$, we conclude
\be 
\lim\limits_{i\to\infty}|\S_i|_{\widetilde g_{\delta_i}}=|\S_H|,
\ee 
where $\S_i$ is the outermost minimal surface in $(M,\widetilde g_{\delta_i})$. Note that the electric charge evaluated on $\S_i$ may be far from the charge on $\S_H$, and $\S_i$ may not be connected. However, the charge at infinity for $(\widetilde g_\delta,(1-\epsilon)\widetilde E_\delta)$ converges to $(1-\epsilon) Q_\infty$ since the metric converges uniformly in $C^0$ and the electric field is only modified on a compact set. Theorem \ref{thm-penrose-corners-A} then follows from taking $\epsilon$ to $0$.

By very similar reasoning we obtain the following lemma analogous to Lemma \ref{lemma_energy_condition}.
\begin{lemma}\label{lem-unif-bound-B}
	Assume $H_- \geq H_+$ and $\Phi_-\geq \Phi_+$, then there exists $\widehat E_\delta$ such that we have
	\be \label{eq-weakDEC}
	R(g_\delta)- 2|\widehat E_\delta|^2_{g_\delta}\geq h_2
	\ee 
	and
	\be 
	\nabla_{g_\delta}\cdot \widehat E_\delta \leq 0
	\ee 
	where $\nabla\cdot \widehat E_\delta$ vanishes outside of $\S\times(-\delta,\delta)$, and as $\delta \to 0$, we have
	\begin{align}
	\|h_2\|_{L^{p}(\widetilde M)}  \to & 0  \label{eq_estimate_1} \\
	Q_\infty( \widehat E_\delta) \to & Q_\infty(E_+) \label{eq_estimate_2},
	\end{align}
	for $p\in\lbrack1,3\rbrack$.
\end{lemma}
\begin{proof}
Consider the same $g_\delta$ and $E_\delta$ as in Lemma \ref{lemma_energy_condition}. Then we have the same lower bound for $ R(g_\delta)- 2|E_\delta|^2$ as in \eqref{eq-weakeDEC}. Moreover, recall that
	\begin{equation}
	\nabla_{g_\delta} \cdot E_\delta =\widetilde h_\delta,\qquad (s,x)\in\left\{\frac{\delta^2}{50}<|s|\leq\frac{\delta}{2}\right\}\times \Sigma,
	\end{equation}
	and
	\begin{equation}
	\nabla_{g_\delta} \cdot E_\delta
	=\left(\Phi_+  -  \Phi_-\right)\frac{50}{\delta^2}
	\phi\left(\frac{100s}{\delta^2}\right)
	+\widetilde h_\delta,\qquad (s,x)\in\left[-\frac{\delta^2}{50},\frac{\delta^2}{50}\right]\times \Sigma.
	\end{equation}
	where $\widetilde h_\delta$ is uniformly bounded on $ M$ and vanishes outside of $\S\times(-\delta,\delta)$. Therefore we again double the manifold $M$ to obtain $\widetilde M$ and reflect $\widetilde h_\delta$ appropriately. On $\widetilde M$, we solve
	\[
	\Delta_{g_\delta}f=\widetilde h_\delta
	\] 
	with $f$ going to $0$ at infinity in both ends. It follows from the reflection symmetry that $\frac{\partial f}{\partial \nu} =0$ at the horizon (cf. Lemma 4 of \cite{Miao09}). We denote by $f_\delta$ the solution.
		
		Now setting
		\be 
		\widehat E_\delta = E_\delta- \nabla_{g_\delta}  f_\delta,
		\ee 
		we have $\nabla_{g_\delta}  \cdot \widehat E_\delta\leq 0$ and $\nabla_{g_\delta}  \cdot \widehat E_\delta$ vanishes identically outside of a compact set. Moreover, 
		\[
		R(g_\delta)- 2|\widehat E_\delta|^2\geq R(g_\delta)- 2|E_\delta|^2 - C_1|E_\delta| |\nabla_{g_\delta}  f_\delta | - C_2|\nabla_{g_\delta}  f_\delta |^2,
		\]
		where we omit the subscript ${g_\delta}$ on the above norms for notational brevity. 
		
             We fix a compact set $K$ containing an open set of the corner $\Sigma$ and derive the necessary estimate on $K$ and $M\setminus  K$ separately. We may assume that $\delta$ is sufficiently small so that $\Sigma \times (-\delta , \delta)$ is contained in $M\setminus  K$. In particular, $\widetilde h_\delta =0$ on $M\setminus  K$ and the metric $g_\delta$ is independent of $\delta$ on $M\setminus  K$.

		On the compact set $K$, we conclude that $f_\delta$ is uniformly bounded in $C^{1,\alpha}$ since $\widetilde h_\delta$ is uniformly bounded in $L^\infty$. On the other hand, integration by parts shows that as $\delta \to 0$, the $L^2$ norm of $\nabla_{g_\delta}  f_\delta$ goes to zero. As a result, $\nabla_{g_\delta}  f_\delta$ goes to zero in $L^p$ on $K$ for $1\le  p<2$ by the simple inequality between $L^p$ norms on compact domains. Moreover, by interpolation $\nabla_{g_\delta}  f_\delta$ goes to zero in $L^p$ on $K$ for any $2\le p < \infty$. 

		On $M\setminus  K$, we therefore want to show that $|\nabla_{g_\delta}  f_\delta|^2$ goes to zero in $L^{p}$ for $p\in\lbrack1,3\rbrack$. This is easiest to see in the language of weighted Sobolev spaces (see \cite{Bartnik-86} for notation and a thorough treatment). In particular Theorem 1.10 of \cite{Bartnik-86} shows that $|\nabla_{g_\delta}  f_\delta|$ is controlled in the weighted $W^{1,2}_{-3/2}$ norm.	We remark that the constant in Theorem 1.10 of \cite{Bartnik-86} depends on the operator $\Delta_{g_\delta}$, however through the use of cut-off functions it depends only on the operator in some exterior region combined with local interior estimates. Since the interior estimates are shown to hold above, and $	\Delta_{g_\delta}$ is a fixed metric away from the corner, we can conclude that the constant in Theorem 1.10 of \cite{Bartnik-86} can be chosen independently of $\Delta_{g_\delta}$.
		
		From this, straightforward applications of the weighted Sobolev inequality (Theorem 1.2(iv) of \cite{Bartnik-86}) show that $|\nabla_{g_\delta}  f_\delta|^2$ goes to zero in the required $L^p$ spaces.

		These estimates give \eqref{eq_estimate_1} and then \eqref{eq_estimate_2} follows by integrating the equation for $f_\delta$ and the divergence theorem.
\end{proof}
Note that the range of permitted values of $p$ in the $L^p$ estimate for $h_2$ includes $6/5,3/2$ and $3$, as required.

It is clear that a version of this lemma holds, assuming instead $\Phi_-\leq \Phi_+$ and concluding instead $\nabla_{g_\delta}\cdot \widehat E_\delta \leq 0$. Then by the same arguments showing that Lemma \ref {lem-conf-changed-data} follows from Lemma \ref{lemma_energy_condition}, we obtain the following two analogous lemmas.

\begin{lemma}
	Suppose  $(M,g_\pm,E_\pm)$ is a manifold with corner such that $H_- \geq H_+$ and $\Phi_-\geq \Phi_+$, and away from the corner \eqref{eq-energy-condition-B} holds. Then there exists a a sequence of initial data $(M,\widetilde g_\delta, (1-\epsilon)^2\widetilde E_\delta)$ satisfying \eqref{eq-energy-condition-B}.
\end{lemma}
\begin{lemma}
	Suppose  $(M,g_\pm,E_\pm)$ is a manifold with corner such that $H_- \geq H_+$ and $\Phi_-\leq \Phi_+$, and away from the corner \eqref{eq-energy-condition-C} holds. Then there exists a a sequence of initial data $(M,\widetilde g_\delta, (1-\epsilon)^2\widetilde E_\delta)$ satisfying \eqref{eq-energy-condition-C}.
\end{lemma}

\subsection*{Proof of the charged Riemannian Penrose inequality with corners}
As above, the proofs of Theorems \ref{thm-penrose-corners-A}, \ref{thm-penrose-corners-B} and \ref{thm-penrose-corners-C} are essentially identical. We therefore simply present the proof of Theorem \ref{thm-penrose-corners-C}.

\begin{proof}[Proof of Theorem \ref{thm-penrose-corners-C}]
Since each $\S_i$ as described above is homologous to $\S_H$, the topological assumption on $\Omega$ guarantees that each $\S_i$ is connected. We can then apply Theorem \ref{thm-C} to conclude
\be 
	\m_{ADM}(\widetilde g_\delta)\geq \(\frac{|\S_{i}|}{16\pi}\)^{\frac12}\(1+\frac{4(1-\epsilon)^4\pi Q(\S_{i})^2}{ |\S_{i}|}\)
\ee 
Since $\nabla_{\widetilde g_\delta}\cdot \widetilde E_\delta\geq0$, by the divergence theorem we have $Q(\S_i)\geq Q(\S_H)$. Taking limits, $\delta$ then $\epsilon$ to zero, we are done.
\end{proof}

\section{Charged asymptotically flat extensions} \label{S-Extensions}
In \cite{Chen18}, a quasi-local Penrose inequality was established for an isometric embedding into a spherically symmetric static manifold that is used as a reference manifold for the quasi-local mass. One of the key ideas is to create a Shi--Tam type extension of the compact initial data with a boundary to obtain an asymptotically flat initial data set using the isometric embedding. To apply the same idea to the charged Penrose inequality, we must show that this construction leads to a charged extension. We then apply the different versions of the charged Penrose inequality with corners to such an extension to obtain each version of the quasi-local charged Penrose inequality.
	
	To this end, we now recall some details from \cite{Chen18} regarding the type of extensions that we will make use of here. Suppose $(M,g)$ is a compact $3$-manifold with boundary $\S$ having non-negative Gauss curvature and non-negative mean curvature. Suppose there exists an isometric embedding $\S_0$ of $\S$ into some Reissner--Nordstr\"om manifold, in which case we can write this Reissner--Nordstr\"om manifold as
	\be 
		\overline{g}=ds^2+\sigma_s,
	\ee 
	where $\sigma_0$ is the induced metric on $\Sigma_0$, and similarly we denote each leaf by $\S_s$ with induced metric $\sigma_s$. Let $\nu=\p_s$ denote the unit normal vector field to the leaves. A function $T$ is defined in \cite{Chen18} via the static equation for a more general static manifold, however here we are only interested in static reference manifolds in the Reissner--Nordstr\"om family so we can fix $T$ explicitly. Specifically, we define $T=2\sin^2(\theta)|E|^2$, where $\theta$ is the angle between the $\p_s$ and $\p_r$; that is, $\cos(\theta)=|\p_r|^{-1}_{\overline g}\p_r\cdot\p_s$. Still following \cite{Chen18}, we are interested in defining a new metric
	\be 
		\widetilde{g}=u^2ds^2+\sigma_s,
	\ee 
	in terms of some warping function $u$. The function $T$ defined above is then used to solve a prescribed scalar curvature equation for the metric $\widetilde g$ to ensure that it satisfies the dominant energy condition, $R(\widetilde{g})\geq 2|\widetilde E|^2$, for an appropriate choice of electric field vector $\widetilde E$. Precisely, the prescribed scalar curvature equation is
	\be \label{eq-prescR}
		R(\widetilde g)=R(\overline g)+(u^{-2}-1)T,
	\ee 
	which can be solved for a smooth asymptotically flat metric $\widetilde g$ for any initial value $u(0)>0$, provided
	\be \label{eq-conv1}
det(A_0)  - \overline Ric  (\nu,\nu)  + \frac{T}{2} > 0,
	\ee 
where $A_0$ denotes the second fundamental form of $\S$, $det(A_0)=\kappa_1\kappa_2$, and $\overline{Ric}$ denotes the Ricci curvature of $\overline g$ (Proposition 4.1 of \cite{Chen18}). Note, throughout we will use the notation $\overline \cdot$ to denote quantities corresponding to $\overline g$, such as $\overline R = R(\overline g)$. We remark that via the Gauss equation, this condition may we expressed as
	\be \label{eq-conv12}
		K>(\Phi)^2.
	\ee 
	Some intuition for this condition can be seen by integrating \eqref{eq-conv12} over $\S$, applying the Gauss-Bonnet Theorem, and comparing this to the area--charge inequality for black holes, $|\S|\geq 4\pi Q^2$ \cite{DJR,Gibbons}.
	
	Having the appropriate asymptotically flat manifolds now at hand, we must ensure that an appropriate choice of the electric field in our extensions can be made. To this end, we decompose the electric field of the reference Reissner--Nordstr\"om manifold, $\overline E$, into
	\be 
		\overline E=\overline E^s\p_s+\overline E^T,
	\ee 
	where $\overline E^T$ is tangential to each leaf, $\S_s$. We then have
\[
\begin{split}
(\overline E^s)^2   = \cos^2 \theta |E|^2 =  \frac{\overline R  \cos^2 \theta}{2} \\
|\overline E^T|^2 = \sin^2 \theta |E|^2 = \frac{\overline R  \sin^2 \theta}{2}.
\end{split}
\]

	Consider the vector field $\widetilde E$ given by
\be \label{eq-Efielddefn}
\widetilde E:=\frac{1}{u} \left ( \overline E^s \p s +\overline E^T \right)
\ee 
and we claim it satisfies
\begin{equation} \label{eq-changed-E}
\begin{split}
R(\widetilde g) = & \, 2 |\widetilde E|_{\widetilde g}^2 \\
\nabla_{\widetilde g}  \cdot \widetilde E = &\, 0. 
\end{split}
\end{equation}

The Christoffel symbols can be computed, noting $\Gamma_{ss}^ i =  0 $, as
\[
\begin{split}
\widetilde \Gamma_{ss}^s &=  - u^{-1} \partial_s u \\
\widetilde \Gamma_{ss}^a &= - u \sigma_s^{ab} \partial_b u\\
\widetilde \Gamma_{ab}^c &= \Gamma_{ab}^c \\
\widetilde \Gamma_{ab}^s &= u^{-2} \Gamma_{ab}^s  .
\end{split}
\]

We then compute
\[
\begin{split}
\nabla_{\widetilde g}  \cdot \widetilde E 
=\, &  u^{-2} \widetilde D_s \widetilde E_s + \sigma_s^{ab} D_a \widetilde E_b \\
=\, & u^{-2} \left ( \partial_s  ( u \overline E^s  ) - \widetilde \Gamma_{ss}^s u\overline E^s  - \widetilde \Gamma_{ss}^a u^{-1} \overline E_a    \right )\\ & + \sigma_s^{ab} \left  ( \partial_a (u^{-1} \overline E_b ) - \widetilde \Gamma_{ab}^c   u^{-1} \overline E_c - \widetilde\Gamma_{ab}^s u \overline E_s \right) \\
=\, & u^{-1} \nabla_{\overline g} \overline E  \\
= \,& 0.
\end{split}
\]

We also see that the energy condition holds, by

\[
\begin{split}
2|\widetilde E|_{\widetilde g}^2 = &2(\overline E^s)^2 + 2u^{-2} |\overline E^T|^2 \\
=&  \overline R ( \cos^2 \theta + u^{-2} \sin^2 \theta ) \\
=&   \overline R  + (\frac{1}{u^2} -1)  \sin^2 \theta \overline R \\
=&  \overline R  + (\frac{1}{u^2} -1) T.
\end{split}
\]

We now turn to consider the quantity
	\be\label{eq-monoquan}
\int_{\Sigma_s} VH_o(1-\frac{1}{u}) d\S_s,
\ee
where $V=\sqrt{1+\frac{\overline Q^2}{r^2}-\frac{2\overline m}{r}}$ is the static potential for the Reissner--Nordstr\"om manifold. We know from Proposition 3.2 of \cite{Chen18}, the quantity defined by \eqref{eq-monoquan} is monotonically non-increasing with respect to $s$ provided that
\be \label{eq-conv3}
det(A_0)- \frac{T}{2} + \frac{\partial V}{\partial \nu} \frac{H_o}{V} > 0.
\ee

In what follows, we will consider separately the cases where the reference Reissner--Nordstr\"om manifold has positive mass or zero mass. In the case of positive mass, we will apply the results of \cite{Chen18} directly to obtain an appropriate extension. However, the hypotheses required for this implicitly assume the mass-charge inequality holds, $\overline m\geq|\overline Q|$ (this is implied by Condition ($\dagger$), below). We therefore treat the case where the reference manifold has zero mass separately. We record here the convexity conditions from \cite{Chen18}.
\begin{definition}\label{def-convex1}
	It was shown in \cite{Chen18} that there exist constants $C_1$ and $C_2$, depending only on the mass and charge of the reference manifold, such that the prescribed scalar curvature equation \eqref{eq-prescR} can be solved, and that the static mass is monotone along the foliation, provided that
\be\label{eq-cond1}\begin{split}
		\overline{\Ric}(\nu,\nu)&<0\\
		\min \kappa_a &> \frac{C_1}{r^2}\\
		r &> C_2,
	\end{split}\ee 
	where $\kappa_a$ are the principal curvatures of $\S_0$ and $r$ standard radial coordinate in the Reissner--Nordstr\"om manifold. If \eqref{eq-cond1} holds for these constants, we say that Condition ($\dagger$) holds.
\end{definition}

We now turn to the case where the reference Reissner--Nordstr\"om manifold has zero mass. Since the quasi-local mass with static reference -- in some sense -- sets its ground state energy to be that of the reference manifold, there is no hope for a positivity result in general. However, by choosing the reference manifold to have zero mass some hope of such a statement is recovered. For this reason, one may view the quasi-local mass with respect to a zero mass Reissner--Nordstr\"om manifold as a kind of ``charged Brown--York mass."

For the Reissner--Nordstr\"om manifold with zero or negative mass, the theorems proven in \cite{Chen18} can not be applied directly. For example, the static potential is now decreasing in the radial outward direction. We now construct appropriate extensions for using the zero mass Reissner--Nordstr\"om manifold as the reference manifold. Recall that the metric is of the form
\[
\overline g = \frac{d r^2}{1  + \frac{\overline Q^2}{r^2}} + r^2 dS^2.
\]

The metric is conformally flat with the following explicit coordinate transformation
\[
r(\rho)=\rho - \frac{\overline Q^2}{4 \rho}.
\]
In terms of $\rho$, the metric reads
\[
\overline g = (1-\frac{\overline Q^2}{4 \rho^2})^2 \left ( d \rho ^2 + \rho^2 dS^2 \right ).
\]
Let $\Sigma_0$ be an image of the isometric embedding into $\overline g$ in the above coordinate system. Let $\Sigma_s$ be the image of the unit normal flow in $\overline g $ starting at $\Sigma_0$. To prove the quasi-local Penrose inequality following the approach in \cite{Chen18}, we need that the $\S_s$ is a convex foliation of the exterior of $\S_0$ such that the prescribed scalar curvature equation \eqref{eq-prescR} admits a smooth solution that converges to $1$ at infinity, such that the quasi-local mass of $\S_s$ with respect to $\widetilde g$ is monotonically non-increasing. As discussed above, we therefore require that \eqref{eq-conv1} and \eqref{eq-conv3} hold:

\[
\begin{split}
det(A_0)- \frac{T}{2} + \frac{\partial V}{\partial \nu} \frac{H_o}{V} >  0 & \\
det(A_0)  - \overline Ric  (\nu,\nu)  + \frac{T}{2} >  0 &.
\end{split}
\]

Recall that $\overline R \ge T \ge 0$. Hence it suffices to have the following three inequalities
\[
\begin{split}
det(A_0)+ 2 \frac{\partial V}{\partial \nu} \frac{H_o}{V} > 0,& \\
det(A_0)- \overline R > 0,& \\
det(A_0)  - \overline Ric  (\nu,\nu)  > 0 &. 
\end{split}
\]
For the static potential, we have 
\[
2\frac{\partial V}{\partial \nu} \frac{1}{V}   \ge - \frac{C_1}{\rho ^3}, 
\]
and for the curvature, we have 
\[
\frac{C_2}{\rho ^4}  \ge \overline R \ge \overline Ric  (\nu,\nu)  > 0.
\]
Hence, it suffices to have that on each $\S_s$,
\[
\rho^3 det(A_0) > C_1 H_o \qquad \text{ and } \qquad \rho^4  det(A_0) > C_2.
\]
We identify $\S_s$ with $\widetilde \S_s$ in $\R^3$ using the conformal transformation. Recall that under a conformal change, the principal curvatures satisfies
\[
\kappa= \frac{\tilde \kappa}{1-\frac{\overline Q^2}{4 \rho^2}}+\frac{\overline Q^2}{2 \rho^3} \frac{\partial \rho}{\partial \nu}.
\]
It follows that $ \rho^4  det(\widetilde A_0) > C_2 $ implies $\rho^4  det(A_0) > C_2$, as long as $ \frac{\partial \rho}{\partial \nu} \ge 0$. Moreover,
\[
\rho^3 det(\widetilde A_0) > C_1 \widetilde H_o \qquad \text{ and } \qquad \rho^3 \widetilde  H_o > 2C_1.
\]
implies that $\rho^3 det(A_0) > C_1 H_o$,  as long as $ \frac{\partial \rho}{\partial \nu} \ge 0$ since
\[
\begin{split}
  & \rho^3 det(A_0) -C_1 H_o \\
= & \rho^3 \left( \frac{\tilde \kappa_1}{1-\frac{\overline Q^2}{4 \rho^2}}+\frac{\overline Q^2}{2 \rho^3} \frac{\partial \rho}{\partial \nu} \right)\left( \frac{\tilde \kappa_1}{1-\frac{\overline Q^2}{4 \rho^2}}+\frac{\overline Q^2}{2 \rho^3} \frac{\partial \rho}{\partial \nu} \right) -C_1\left( \frac{\widetilde H_o}{1-\frac{\overline Q^2}{4 \rho^2}}+\frac{\overline Q^2}{ \rho^3} \frac{\partial \rho}{\partial \nu} \right)\\
\ge & \frac{\rho^3 det(\widetilde A_0) - C_1 \widetilde H_o}{(1-\frac{\overline Q^2}{4 \rho^2})^2} + \frac{1}{2}\frac{(\rho^3 \widetilde  H_o - 2C_1)}{1-\frac{\overline Q^2}{4 \rho^2}}\frac{\overline Q^2}{ \rho^3} \frac{\partial \rho}{\partial \nu}+ \rho^3 (\frac{\overline Q^2}{2 \rho^3} \frac{\partial \rho}{\partial \nu})^2.
\end{split}
\]
As a result, it suffices to require that
\begin{equation}\label{convex_condition_zero_mass_flow}
\frac{\partial \rho}{\partial \nu} > 0  \qquad \rho^3 det(\widetilde A_0) > C_1 \widetilde H_o \qquad  \rho^3 \widetilde  H_o > 2C_1 \qquad \text{ and } \qquad \rho^4  det(\widetilde A_0) > C_2 
\end{equation}
on $\widetilde \S_s$ where $\widetilde A_0$ and $ \widetilde H_o$ are the second fundamental form and mean curvature of $\widetilde \S_s$  in $\R^3$. 

Recall that the family $\widetilde \S_s$  in $\R^3$ satisfies the normal flow with the prescribed speed 
\[
F = \frac{1}{1-\frac{\overline Q^2}{4\rho^2}},
\]
coming from the conformal factor. Then along the flow, $\rho$ and $\cos \theta$ satisfy
\[
\begin{split}
\frac{d}{ds} \rho & = F \cos \theta \\
\frac{d}{ds}  \cos \theta & \ge \frac{F \sin^2 \theta }{ 2 \rho}.
\end{split}
\]
Hence, $\rho$ and  $\cos \theta$ are both increasing as long as they are positive initially. The second fundamental form satisfies (see \cite{Zhu}, for example)
\begin{equation}\label{2ff_evol}
\frac{d}{ds}  (\widetilde A_0)_{ab} = -  \nabla_a \nabla_b F  -   F A_{ac} \sigma^{cd}(\widetilde A_0)_{db}
\end{equation}
where
\[
- \nabla_a \nabla_b F= - D_a D_b F+ A_{ab}  \partial_\nu F,
\]
and $a,b=1,2$ are coordinates on each leaf. There exists constants $C_3$ and $C_4$ such that for $\rho > Q$
\[
\frac{C_3}{\rho^4}  \ge | D_a D_b F |  
\]
and 
\[
0 \ge \partial_\nu F \ge -\frac{C_4}{\rho^3}.
\]
Combining \eqref{2ff_evol} with the above bounds, we get, by diagonalizing $\widetilde A_0$ at a given time,
\[
 - F \widetilde \kappa^2 + \frac{C_3}{\rho^4}\ge \frac{d}{ds} \widetilde \kappa  \ge - F \widetilde \kappa^2-  \frac{\left ( C_3 + C_4 \widetilde \kappa \rho \right)}{\rho^4}.
\]
Thus, we have 
\[
\rho^4 \frac{d}{ds} \widetilde H_o =\rho^4 \frac{d}{ds} (\widetilde \kappa_1 +  \widetilde \kappa_2) \le  -F\rho^4 (\widetilde \kappa_1^2 +  \widetilde \kappa_2^2) + 2C_3< - \rho^4 det(\widetilde A_0) +C_3.
\]
This shows that $\widetilde H_o$ is decreasing assuming $ \rho^4 det(\widetilde A_0) > C_3$. Let $\sqrt{C'} = \max\{C_2,C_3\}$. \eqref{convex_condition_zero_mass_flow} holds for all $s$ if 
\begin{equation}\label{convex_condition_one_mass_flow}
 \rho^3 \widetilde  H_o > 2C_1 \qquad \text{ and } \qquad \rho^4  det(\widetilde A_0) >( C' )^2
\end{equation}
hold for all $s$ and $\cos \theta >0$ and $ \rho^3 det(\widetilde A_0) > C_1 \widetilde H_o$ hold  for $s=0$.

Since $\rho^4  det(\widetilde A_0) > (C')^2  $  follows from $ \rho^2 \widetilde \kappa > C'$, we compute
\[
\begin{split}
\frac{d}{ds} \rho^2 \widetilde \kappa \ge & \frac{F}{\rho} \left ( 2\cos\theta \rho^2 \widetilde  \kappa  - \frac{(\rho^2 \widetilde  \kappa)^2}{\rho} -  \frac{C_3}{\rho} - C_4 \widetilde \kappa \right) \\
=&  \frac{F}{\rho} \left (\rho^2 \widetilde  \kappa \Big ( 2\cos\theta  - \frac{\rho^2 \widetilde  \kappa}{\rho} - \frac{C_4}{\rho^2} \Big )   -  \frac{C_3}{\rho} \right) 
\end{split}
\]
Assuming $\cos \theta > \delta$ for some $\delta >0$ when $s=0$, then $\cos \theta > \delta$ for all $s$. We claim that there exists $C$ and $C'$  such that $ \rho^2 \widetilde \kappa > C'$  for all $s$ if
$ \rho^2 \widetilde \kappa > C'$ and $\rho> C$ at $s=0$. We prove by contradiction.

Suppose it does not hold for all $s$. Then there is a first time $s= s_0$ that it is violated. At $s = s_0$, we have
\[
 0 \ge \frac{d}{ds} \rho^2 \widetilde \kappa   >  \frac{F}{\rho}  \left ( C'\Big ( 2\delta  -\frac{ C'}{C}- \frac{C_4}{C^2} \Big )   -  \frac{C_3}{C}   \right) 
\]
the right hand side is positive for sufficiently large $C$.

Finally, we observe that $\rho^3 \widetilde  H_o > 2C_1 $  and $ \rho^3 det(\widetilde A_0) >C_1 \widetilde H_o$ both follow from $ \rho^2 \widetilde \kappa > C'$ and $\rho> C$. Indeed, 
\[
\begin{split}
\rho^3 \widetilde  H_o =& \rho (\rho^2 \sum_a \tilde \kappa_a) > 2CC' \\
\rho^3 det(\widetilde A_0) =&\frac{1}{2} \rho  (\rho^2 \tilde \kappa_1\tilde \kappa_2 + \rho^2 \tilde \kappa_2\tilde \kappa_1 ) > \frac{1}{2}CC' \widetilde  H_o.
\end{split}
\]
As a result, it suffices to choose $C$ such that 
\[
CC'=2C_1.
\]

Then from Propositions 3.2 and 4.1 of \cite{Chen18}, the proof of Proposition 4.3 therein, and the discussion given above, we conclude:
\begin{prop}\label{prop-cond2}
       Given $\delta>0$, there exist constants $C,C'$ depending only on $\delta$ and the charge of the reference manifold such that for any surface $\S_0$ in the zero mass Reissner--Nordstr\"om manifold if
	\begin{equation}\label{convex_condition_zero_mass_initial}
	\begin{split}
     \cos \theta >& \delta   \\
     \rho > & C \\
     \rho^2 \widetilde \kappa >& C' 
	\end{split}
	\end{equation}
	then the unit normal flow from $\S_0$ creates a convex foliation of the exterior of $\S_0$. Furthermore, for any initial value $u>0$ on  $\S_0$, the solution $\widetilde g$ to the prescribed scalar curvature equation \eqref{eq-prescR} is a smooth asymptotically flat metric such that 
	\[
	\int_{\S_0} V H_o (1- \frac{1}{u}) d\S_0 \ge \m_{ADM}(\widetilde g) .
	\]
\end{prop}
\begin{definition}\label{def-convex2}
	If \eqref{convex_condition_zero_mass_initial} holds on $\S_0$ for those constants given by Proposition \ref{prop-cond2}, then we say that $\S_0$ satisfies the convexity condition ($\dagger\dagger$).
\end{definition}

	\section{Quasi-local charged Penrose inequalities}\label{S-proof}
We are now in a position to prove the quasi-local Penrose inequalities by gluing extensions of the form constructed above to a compact manifold with boundary. We first prove Theorem \ref{thm-QLPenrose-A}, and note that the proof is identical regardless of whether or not we impose convexity condition ($\dagger$) and $\overline m >0$, or convexity condition ($\dagger\dagger$) and $\overline m=0$.

	\begin{proof}[Proof of Theorem \ref {thm-QLPenrose-A}]
Let $\widetilde g$ be the metric constructed above using the foliation, where the boundary conditions for $u$ are yet to be specified.
		
		By the reasoning given in the preceding section, for any initial value $u>0$ on $\S_0$, the solution to the prescribed scalar curvature equation is a smooth solution that converges to $1$ as $s$ goes to infinity. Moreover, the quantity
		\be
		\int_{\Sigma_s} VH_o(1-\frac{1}{u})d\S_s
		\ee
		monotonically decreases to $\m_{ADM}(\widetilde g) - \overline m$ where $\overline m$ is the mass of the reference Reissner--Nordstr\"om manifold. Taking $\widetilde E$ to be the vector field 
		\bee
		\frac{1}{u} \left ( \overline E^s \frac{\partial}{\partial s} + \overline  E^T \right)
		\eee 
		as constructed above, where $\overline\cdot$ indicates quantities in the reference Reissner--Nordstr\"om manifold. Let $M_{ext}$ be the portion of the Reissner--Nordstr\"om manifold exterior to $\S_0$. The initial data $(M_{ext},\widetilde g,\widetilde E)$ then satisfies
		\begin{equation}
		\begin{split}
		R(\widetilde g) = & 2 |\widetilde E|_{\widetilde g}^2 \\
		\nabla_{\widetilde g}  \cdot \widetilde E = & 0. 
		\end{split}
		\end{equation}
		Moreover, the normal flux of $\widetilde E$ is independent of the choice of $u$ -- that is, $\widetilde \Phi = \overline \Phi$. As a result, assuming
		\be
		H > 2 | \overline\Phi -  \Phi | 
		\ee
		we may choose
		\[
		u= \frac{H_o}{H -2| \overline\Phi -  \Phi| }.
		\]
		Note that the mean curvature of $\p M_{ext}$ (with respect to $g$), is given by 
		\be 
		\widetilde H=\frac{H_o}{u}=H-2 | \overline\Phi -  \Phi |>0.
		\ee
		We remind the reader that at this point there are 4 distinct mean curvatures coming into play, $\widetilde H_o,\widetilde H, H$ and $H_o$.
		 
		We readily see that if we glue together $(\Omega,g,E)$ with $(M_{ext},\widetilde g,\widetilde E)$, it forms a charged manifold admitting corner along $\S$ such that the hypotheses of Theorem \ref{thm-penrose-corners-A} are satisfied. It follows that 	
		\be 
		\m_{ADM}\geq\( \frac{|\S_H|}{16\pi}\)^{1/2}\( 1+\frac{4\pi \overline Q^2}{|\S_H|} \)
		\ee 
		where $\overline Q$ is the charge of the reference Reissner--Nordstr\"om manifold.
		
		The monotonicity of quasi-local mass implies that 
		\be
		\m_{ADM}(\widetilde g) - \overline m \le  \int_{\Sigma_0} VH_o(1-\frac{1}{u})d\S_0 =\int_{\S}V(H_o-H+2 | \overline\Phi -  \Phi|  )d\S,
		\ee
		completing the proof.
	\end{proof}

By a closely related argument, we are also able to establish Theorems \ref{thm-QLPenrose-B} and \ref{thm-QLPenrose-C}.

\begin{proof}[Proof of Theorem \ref{thm-QLPenrose-B}]
	We take the same foliation as in the preceding proof, except we will choose a different boundary condition for $u$. Here we simply choose $u=\frac{H_o}{H}$ so that the mean curvature on the boundary of the constructed extension is $\widetilde H=H$. By hypothesis $\overline \Phi \leq \Phi$ on $\S$ and $\widetilde E$ is constructed to be divergence-free, so we apply Theorem \ref{thm-penrose-corners-B} to obtain
		\be 
\m_{ADM}\geq\( \frac{|\S_H|}{16\pi}\)^{1/2}\( 1+\frac{4\pi \overline Q^2}{|\S_H|} \)
\ee 
where $\overline Q$ is the charge of the reference Reissner--Nordstr\"om manifold.

The monotonicity of quasi-local mass implies that 
\be
\m_{ADM}(\widetilde g) - \overline m \le  \int_{\Sigma_0} VH_o(1-\frac{1}{u}) d\S_0=\int_{\S}V(H_o-H)\,d\S,
\ee
completing the proof.
\end{proof}
\begin{proof}[Proof of Theorem \ref{thm-QLPenrose-C}]
The proof is essentially identical to the above, except that Theorem \ref{thm-penrose-corners-C} is applied instead of Theorem \ref{thm-penrose-corners-B}, yielding the charge on $\S_H$ rather than at infinity.
\end{proof}
\begin{remark} \label{rem-equality}
By Corollary 3.3 of \cite{Chen18} one sees that $u\equiv 1$ in the case of equality for each of these three theorems. This in turn implies that $H\equiv H_o$ in Theorems \ref{thm-QLPenrose-B} and \ref{thm-QLPenrose-C}, and $H_o-H+2|\overline \Phi - \Phi|\equiv 0$ in Theorem \ref{thm-QLPenrose-A}. Furthermore, one then sees that if the reference Reissner--Nordstr\"om manifold is chosen to have zero mass, then this implies $|\S_H|=0$. By the assumption that $\S_H$ be non-empty, we conclude that analogous to the quasi-local Penrose inequality for the Brown--York mass (cf. \cite{LM}), these inequalities are in fact strict when $\overline m$ is chosen to be zero.
\end{remark}

\end{document}